\documentclass[twoside,a4paper,11pt]{article}
\usepackage[all]{xy}

\usepackage{graphicx}
\usepackage[latin1]{inputenc}
\usepackage{amsmath,amsthm}
\usepackage{amsfonts,amssymb}
\usepackage{bigdelim}
\usepackage{multirow}
\usepackage{enumerate}
\usepackage{stmaryrd}
\usepackage{float}
\usepackage{pstricks}
\usepackage[ae]{}
\usepackage{lineno}

\newcommand{\superscript}[1]{\ensuremath{^{\textrm{#1}}}}

\def\wg{\superscript{\dag}} 
 
\newcommand{\E}{\mathbb{E}}
\newcommand{\N}{\mathbb{N}}
\newcommand{\Z}{\mathbb{Z}}

\newcommand{\R}{\mathbb{R}}

\newcommand{\Pb}{\mathbb{P}}

\def\={{\;\mathop{=}\limits^{\text{(law)}}\;}}

\newtheorem{theorem}{Theorem}[section]
\newtheorem{prop}[theorem]{Proposition}
\newtheorem{lemma}[theorem]{Lemma}
\newtheorem{defi}[theorem]{Definition}
\newtheorem{corol}[theorem]{Corollary}

\theoremstyle{definition}
\newtheorem{rem}[theorem]{Remark}

\newtheorem{exa}[theorem]{Example}

\numberwithin{equation}{section}

\usepackage{fancyhdr}
\date{}

\usepackage[english]{babel}

\title{MRL order, log-concavity and an application to peacocks}
\author{ANTOINE MARIE BOGSO\wg\footnote{
Email address: ambogso@gmail.com, Tel.: +33(0)620150780.}\\
{\wg}Université de Lorraine\\
Institut Elie Cartan de Lorraine, UMR 7502\\
Vandoeuvre-lès-Nancy, F-54506, France.
}

\begin{document}
\maketitle
\noindent
{\footnotesize  {\bf Abstract: }
We provide an equivalent log-concavity condition to the mean residual life (MRL) ordering for real-valued processes. This result, combined with classical properties of total positivity of order 2, allows to exhibit new families of integrable processes which increase in the MRL order (MRL processes). Note that MRL processes with constant mean are   peacocks to which the Azéma-Yor (Skorokhod embedding) algorithm yields an explicit associated martingale.}

\chead[\small ANTOINE MARIE BOGSO]{MRL order, log-concavity and an application to peacocks}

\section{Introduction}
The main objective of this paper is to exhibit new integrable processes which increase in the MRL order (MRL processes). For this end, we exploit the link between the MRL ordering and the notion of total positivity of order 2 to obtain several transformations that preserve the MRL ordering. Moreover, MRL processes with constant mean are peacocks and the connection between the MRL ordering and the Azéma-Yor algorithm allows to construct explicitly an associated martingale to each of them. \\  
Unless otherwise expressly mentioned, the processes and random variables in the sequel are supposed to be real-valued.

\subsection{Definition of the MRL order}
Let $X$ be an integrable random variable and let $\mu$ denote the law of $X$. The MRL function $L_{\mu}$ of $X$ may be defined on $\R$ as
$$
L_{\mu}(x)=\left\{
\begin{array}{ll}
\dfrac{1}{\mu([x,+\infty[)}
\displaystyle\int_{[x,+\infty[}(y-x)\mu(dy)&\text{if }x<b_{\mu},\\ &\\
0&\text{otherwise,}
\end{array}
\right.
$$
where $b_{\mu}=\inf\{z\in\R,\,\mu([z,+\infty[)=0\}$, $b_{\mu}\in(-\infty,+\infty]$. The MRL function is usually of interest for a nonnegative random variable. For instance, if $X$ is thought of as the lifetime of a device, then, for every $x\geq0$, $L_{\mu}(x)$ expresses the conditional expected residual life of the device at time $x$ given that the device is still alive at time $x$. But, there is no restriction on the support of $X$ here. Note that $L_{\mu}$ is positive, left-continuous and such that $x\longmapsto L_{\mu}(x)+x$ is non-decreasing. Moreover, if there exists some $x_0$ such that $L_{\mu}(x_0)=0$, then, for every $x\geq x_0$, $L_{\mu}(x)=0$. Otherwise, we have 
$$
\int_{0}^{+\infty}\frac{dx}{L_{\mu}(x)}=\infty.
$$  
We refer to Shaked-Shanthikumar \cite[Chapter 2]{ShS} for further interesting properties of MRL functions.

Let $X_1$ and $X_2$ be two integrable random variables. Let $\mu_1$, resp. $\mu_2$ denote the law $X_1$, resp. $X_2$. Then, $X_1$ is said to be smaller than $X_2$ in the MRL order if, 
$$
\forall\,x\in\R,\,L_{\mu_1}(x)\leq L_{\mu_2}(x).
$$ 
As many stochastic orders, the MRL order may provide good approximations and close bounds in situations where realistic stochastic models are too complex for rigorous computations. Here, the MRL ordering plays a quite different role. It is shown that the MRL order is also helpful in the construction of peacock processes and their associated martingales.  

\subsection{An application of the MRL ordering to peacocks}
\begin{defi}\label{defi:Pcoc}
A  process $(X_t,t\geq0)$ is said to be a peacock if it is integrable, i.e. 
$$
\forall\,t\geq0,\,\text{ }\E[|X_t|]<\infty,
$$
and if it increases in the convex order, i.e. for every convex function $\psi:\R\to\R$,
$$
t\in\R_+\longmapsto\E[\psi(X_t)]\in]-\infty,+\infty]\,\text{ is non-decreasing.}
$$  
\end{defi}
\noindent
To prove that an integrable process with constant mean is a peacock, we may restrict ourselves to the set of convex functions $z\longmapsto(z-x)^+$, $x\in\R$. This is the purpose of the next result taken from \cite[Section 3.A]{ShS}.
\begin{prop}\label{prop:PcocEquivdTail}
Let $(X_t,t\geq0)$ be an integrable  process such that $\E[X_t]$ does not depend on $t$. The following assertions are equivalent.
\begin{enumerate}
\item For every convex function $\psi:\R\to\R$, $t\longmapsto\E[\psi(X_t)]$ is non-decreasing.
\item For every $x\in\R$, $t\longmapsto\E[(X_t-x)^+]$ is non-decreasing.
\end{enumerate}
\end{prop}
\noindent 
The peacock property involves only one-dimensional marginals. For this reason, we may alternatively define a peacock as a family of probability measures which increases in the convex order.
\begin{defi}
A family $(\mu_t,t\geq0)$ of probability measures on $\R$  is said to be a peacock if it is integrable, i.e.
$$
\forall\,t\geq0,\,\int_{\R}|y|\mu_t(dy)<\infty, 
$$
and if it increases in the convex order, i.e. for every convex function $\psi:\R\to\R$,
$$
t\longmapsto\int_{[x,+\infty[}\psi(y)\mu_t(dy)\,\text{ is non-decreasing.}
$$ 
\end{defi}

It is a direct consequence of Jensen's inequality that every martingale is a peacock. Conversely, a remarkable result due to Kellerer \cite{Kel} states that, for a given peacock $(X_t,t\geq0)$, there exists a martingale $(M_t,t\geq0)$ which
may be chosen Markovian such that 
\begin{equation}\label{eq:EqualInLaw}
\forall\,t\geq0,\,\text{ }M_t\=X_t.
\end{equation}   
If (\ref{eq:EqualInLaw}) holds, we shall say that the martingale $(M_t,t\geq0)$ is {\it associated} to the peacock $(X_t,t\geq0)$. Let us mention that  Hirsch-Roynette \cite{HR} provided recently an alternative proof of the Kellerer's theorem. 
However, since the proof  given by Kellerer is not constructive, it does not allow to obtain concretely an associated martingale to a given peacock. This is what motivated  the authors of \cite{BY, HPRY, MY} who exhibited several examples of peacocks and give some methods to construct explicitly associated martingales to many of them. One method studied in \cite{HPRY, MY} is the Azéma-Yor embedding algorithm which we now briefly describe. Let $(X_t,t\geq0)$ be a centered peacock. In particular, for every $t\geq0$, $\E[|X_t|]<\infty$ and $\E[X_t]=0$. For every $t\geq0$, we denote by $\mu_t$ the law of $X_t$. The Hardy-Littlewood function $\Psi_{\mu_t}$ of $\mu_t$ is given by:
$$
\Psi_{\mu_t}(x)=\left\{
\begin{array}{ll}
\dfrac{1}{\mu([x,+\infty[)}\displaystyle\int_{[x,+\infty[}y\mu(dy)&\text{if }x<b_{\mu_t},\\&\\
x&\text{otherwise,}
\end{array}
\right.
$$ 
where $b_{\mu_t}=\inf\{z\in\R,\,\mu([z,+\infty[)=0\}$. Observe that, if $L_{\mu_t}$ denotes the MRL function of $X_t$, then
$$
\forall\,x\geq0,\,\Psi_{\mu_t}(x)=L_{\mu_t}(x)+x.
$$
The Azéma-Yor solution to the Skorokhod embedding problem for $\mu_t$ (see e.g. \cite{AY} or \cite[Section 5]{Obl}) is the stopping time
$$
T_{\mu_t}=\inf\{v\geq0,\,S_v\geq\Psi_{\mu_t}(B_v)\},
$$
where $S_v=\sup\limits_{0\leq s\leq v}B_s$. If
\begin{equation}\label{eq:AYcondMart}
t\longmapsto T_{\mu_t}\,\text{ is a.s. non-decreasing,}
\end{equation}
then $(M_t:=B_{T_{\mu_t}},t\geq0)$ is a martingale associated to $(X_t,t\geq0)$. Moreover, it is proved in Madan-Yor \cite[Theorem 2]{MY} that $(M_t,t\geq0)$ is an inhomogeneous Markov process.  
But, the condition (\ref{eq:AYcondMart}) is equivalent to:
\begin{equation}\label{eq:AYcondMartII}
\forall\,x\in\R,\,\text{ }t\longmapsto\Psi_{\mu_t}(x)\text{ is non-decreasing.}
\end{equation}
Therefore, it remains to find sufficient conditions on $(X_t,t\geq0)$ under which (\ref{eq:AYcondMartII}) holds.
The condition ({\ref{eq:AYcondMartII}) means that $(X_t,t\geq0)$ increases in the MRL order. We mention that several examples of MRL ordered peacocks are given in \cite[Section 7.4]{HPRY} and in \cite[Section 3]{MY}. Further interesting classes of  MRL ordered peacocks are exhibited in Section 4 below.

\subsection{The aim and organization of this paper} 
We first prove the equivalence between the MRL ordering and a total positivity of order 2 property. Then, we use this result and some classical log-concavity properties to exhibit new families of MRL processes.    
 
The remainder of this paper is structured as follows. In the next section, we present some nice properties of $\R_+$-valued Markov processes which have totally positive transition kernels. 
In section 3, we give an equivalent condition to the MRL ordering using the notion of total positivity of order 2. Precisely, we prove that an integrable process $(X_t,t\geq0)$  increases in the MRL order if and only if its {\it integrated survival function} $C$ defined by:
$$
\forall\,(t,x)\in\R_+\times\R,\,C(t,x)=\E[(X_t-x)^{+}]
$$
is totally positive of order 2, i.e. for every $0\leq t_1\leq t_2$ and $x_1\leq x_2$,
$$
\det\begin{pmatrix}
C(t_1,x_1)&C(t_1,x_2)\\
C(t_2,x_1)&C(t_2,x_2)
\end{pmatrix}\geq0.
$$
Finally, in section 4, we exploit the previous equivalence and total positivity results to provide  new families of MRL ordered peacocks.    
 
\section{Log-concavity properties of nonnegative Markov processes}
We present some results due to Karlin \cite{KA} about $\R_+$-valued Markov processes which have a totally positive of order 2 transition kernel. Note that this family of Markov processes includes $\R_+$-valued processes with independent and log-concave increments, absolute values of processes with independent and symmetric PF$_{\infty}$ increments (see Example \ref{exa:PFfunctions} below for the definition of a PF$_{\infty}$ random variable), birth-death processes and $\R_+$-valued diffusions (we refer to \cite{KA} for more interesting examples). We start with some basic definitions and results.
\begin{defi}\label{defi:TP2Funct}
Let $I$ and  $J$ be subsets of $\R$ such that each is either an interval or a subset of the set of all integers (denoted by $\Z$). A function $p:I\times J\to \R_+$ is said to be totally positive of order 2 (TP$_2$) if, for every $x_1\leq x_2$, elements of $I$, and every $y_1\leq y_2$, elements of $J$,
\begin{equation}\label{eq:DefRTP2}
p
\begin{pmatrix}
x_1,x_2\\
y_1,y_2
\end{pmatrix}
:=\det
\begin{pmatrix}
p(x_1,y_1)&p(x_1,y_2)\\
p(x_2,y_1)&p(x_2,y_2)
\end{pmatrix}
\geq0.
\end{equation}
\end{defi}

\begin{rem}
Let $I$ and $J$ be given as in Definition \ref{defi:TP2Funct}. Then, for every function $p:I\times J\to\R_+$, the following assertions are equivalent.
\begin{enumerate}
\item $p$ is TP$_2$,
\item The set $D=\{(x,y)\in I\times J;\,p(x,y)>0\}$ is a sublattice of $I\times J$, i.e. $D$ satisfies 
\begin{equation*} 
\left.
\begin{array}{l}
\text{For every }x_1<x_2\text{ in }I\text{ and }y_1<y_2\text{ in }J,\text{ } (x_1,y_2)\in D\text{ and }\\ \\
   (x_2,y_1)\in D\text{ imply }
 (x_1,y_1)\in D\text{ and }(x_2,y_2)\in D 
\end{array}
\right\}
\end{equation*}
and $p$ is TP$_2$ on $D$.
\end{enumerate}
\end{rem}
\begin{exa}\label{exa:PFfunctions}
Let $X$ be a log-concave random variable, i.e. $X$ admits a probability density $p:\R\to\R_+$,
the set $S_p:=\{x\in\R:\,p(x)>0\}$ is an interval of $\R$ and $p$ is log-concave on $S_p$. Then, according to  Daduna-Szekli \cite[Theorem 4, Point 1)]{DS}, $p$ is a P\'olya frequency function of order 2 (PF$_2$ function), i.e. for every $x_1\leq x_2$ and every $y_1\leq y_2$,
\begin{equation}\label{eq:PF2}
\det\begin{pmatrix}
p(x_1-y_1)&p(x_1-y_2)\\
p(x_2-y_1)&p(x_2-y_2)
\end{pmatrix}\geq0
\end{equation}
which means that $(x,y)\longmapsto p(x-y)$ is TP$_2$ on $\R\times\R$. Many common random variables are log-concave. Indeed, Gaussian, uniform, exponential, binomial, negative binomial, geometric and Poisson  random variables are log-concave. There are also many common random variables which are not log-concave, for example those with heavy tailed densities. More generally, one calls P\'olya frequency function of order $m$ (PF$_m$ function), where $m\in\N^{\ast}$, a Lebesgue-integrable function $p:\R\to\R$ such that, for every $l\in\{1,\cdots,m\}$ and every real numbers $x_1\leq\cdots\leq x_l$, $y_1\leq\cdots\leq y_l$,
\begin{equation}\label{eq:PFm}
\det\left[(p(x_i-y_j))_{1\leq i,j\leq l}\right]\geq0.
\end{equation}
Observe that PF$_1$ functions are nonnegative functions.
A P\'olya frequency function (PF$_{\infty}$ function) is a function which is PF$_m$ for every $m\in\N^{\ast}$. A random variable is said to be PF$_m$, resp. PF$_{\infty}$ if $X$ admits a PF$_m$, resp. PF$_{\infty}$ probability density $p$. Schoenberg \cite{Sch} provided a characterization of PF$_{\infty}$ functions in terms of their Laplace transform. The discrete counterpart of Schoenberg's characterization was obtained by Edrei \cite{Ed}. Thanks to these results Karlin \cite{KA} proved that if $p$ is a symmetric PF$_{\infty}$ probability density, then the function $\mathbf{p}:\R_+\times\R_+\to\R_+$ defined by
$$
\mathbf{p}(x,y)=p(-x-y)+p(-x+y)
$$
is TP$_2$, i.e. for every nonnegative real numbers $x_1\leq x_2$, $y_1\leq y_2$, 
$$
\det\begin{pmatrix}
\mathbf{p}(x_1,y_1)&\mathbf{p}(x_1,y_2)\\
\mathbf{p}(x_2,y_1)&\mathbf{p}(x_2,y_2)
\end{pmatrix}\geq0.
$$
\end{exa}
The following result may be deduced from the classical Cauchy-Binet formula (see e.g. \cite[Chapter 0]{KAa} or \cite[Problem 11.1.28]{PS}). It is also a particular case of a famous result due to Prékopa \cite[Paragraph 1.1]{Ma}.  
\begin{prop}\label{prop:TPCompForml}
Let $I$, $J$ and  $K$ be subsets of $\R$ such that each is either an interval or a subset of $\Z$ and let $\sigma$ denote a positive measure on $J$. Let $p:I\times J\to\R_+$ and $q:J\times K\to\R_+$ be two TP$_2$ functions such that the product $r:I\times K\to\R_+$ given by:
$$
\forall\,x,z\in\R,\text{ }r(x,z)=\int_{\R}p(x,y)q(y,z)\sigma(dy)
$$
is finite. Then $r$ is TP$_2$.\\ In particular, if $I$ denotes either $\R$ or $\Z$ and if $p,q:I\to\R_+$  are two integrable log-concave functions, then their convolution product $r=p\ast q$ is also log-concave. 
\end{prop}
The previous definition and results apply to the transition kernels of one-dimensional Markov processes.
\begin{defi}
Let $P:=(P_{s,t}(k,l),0\leq s<t,(k,l)\in I\times I)$ be the transition function of a continuous time Markov chain $(\Lambda_t,t\geq0)$ whose state space is a sub-interval $I$ of $\Z$. We say that $P$ is TP$_2$ if, for every $0\leq s<t$ and every $k_1\leq k_2$, $l_1\leq l_2$ elements of $I$,
\begin{equation}\label{eq:defiMarkovZTP2}
P_{s,t}\left(
\begin{array}{cc}
k_1,k_2\\
l_1,l_2
\end{array}
\right):=\det\left(
\begin{array}{cc}
P_{s,t}(k_1,l_1)&P_{s,t}(k_1,l_2)\\ &\\
P_{s,t}(k_2,l_1)&P_{s,t}(k_2,l_2)
\end{array}
\right)\geq0.
\end{equation}
If $(\Lambda_t,t\geq0)$ is homogeneous, then (\ref{eq:defiMarkovZTP2}) is equivalent to:
\begin{equation}\label{eq:defiMarkovHmgeneZTP2}
P_{t}\left(
\begin{array}{cc}
k_1,k_2\\
l_1,l_2
\end{array}
\right):=\det\left(
\begin{array}{cc}
P_{t}(k_1,l_1)&P_{t}(k_1,l_2)\\ &\\
P_{t}(k_2,l_1)&P_{t}(k_2,l_2)
\end{array}
\right)\geq0.
\end{equation}
\end{defi}
We also give a counterpart of the preceding definition for Markov processes with continuous state space.
\begin{defi}
Let $P:=(P_{s,t}(\theta,d\lambda),0\leq s<t,\theta\in I)$ be the transition kernel of a Markov process $((\Lambda_t,t\geq0),(\Pb_{\theta},\theta\in I))$ whose state space is a sub-interval $I$ of $\R$. $P$ is said to be {\em totally positive of order 2} (TP$_2$) if, for every $0\leq s<t$, every $\theta_1\leq\theta_2$ elements of $I$, and every Borel subsets $E_1$, $E_2$ of $I$ such that $E_1\leq E_2$ (i.e. $a_1\leq a_2$ for every $a_1\in E_1$ and $a_2\in E_2$), we have:
\begin{equation}\label{eq:defMarkovTP2}
P_{s,t}\left(
\begin{array}{cc}
\theta_1,\theta_2\\
E_1,E_2
\end{array}
\right):=\det\left(
\begin{array}{cc}
P_{s,t}(\theta_1,E_1)&P_{s,t}(\theta_1,E_2)\\ &\\
P_{s,t}(\theta_2,E_1)&P_{s,t}(\theta_2,E_2)
\end{array}
\right)\geq0.
\end{equation}
Suppose moreover that $(\Lambda_t,t\geq0)$ is homogeneous. Then $P$ is TP$_2$ if and only if 
\begin{equation}\label{eq:defHmgeneMarkovTP2}
P_{t}\left(
\begin{array}{cc}
\theta_1,\theta_2\\
E_1,E_2
\end{array}
\right):=\det\left(
\begin{array}{cc}
P_{t}(\theta_1,E_1)&P_{t}(\theta_1,E_2)\\ &\\
P_{t}(\theta_2,E_1)&P_{t}(\theta_2,E_2)
\end{array}
\right)\geq0.
\end{equation} 
\end{defi}
\begin{rem}
Let $P:=(P_{s,t}(\theta,d\lambda),0\leq s<t,\theta\in I)$ denote the transition kernel of a Markov process $((\Lambda_t,t\geq0),(\Pb_{\theta},\theta\in I))$ which takes values in a sub-interval $I$ of $\R$. Suppose that, for every $0\leq s<t$ and $\theta\in I$, $P_{s,t}(\theta,d\lambda)=p_{s,t}(\theta,\lambda)d\lambda$, where $p_{s,t}$ is continuous. Then $P$ is TP$_2$ if and only if $p_{s,t}$ is TP$_2$. 
\end{rem}
There are many interesting Markov processes which admit a TP$_2$ transition kernel. For example, processes with independent and log-concave increments, absolute values of processes with independent and symmetric PF$_{\infty}$ increments, birth-death processes, one-dimensional diffusions and the bridges of one-dimensional diffusions have  a TP$_2$ transition kernel (see e.g., \cite{Bo, KA, KaM, KAa}). 

Now, we restrict our attention to $\R_+$-valued Markov processes which admit a TP$_2$ transition kernel. In \cite{KA}, the author derives another type of totally positive kernels from these stochastic processes. These kernels have the particularity of involving variables one of which corresponds to the time while the other stands for the state of the process. The following results are simplified versions of those given in \cite{KA}.  
\begin{theorem}\label{theo:TP2MarkovChain}(Karlin \cite[Theorems 2.6 (2.16) and 4.3 (i)]{KA} ). 
\item[1)]Let $(\Lambda_n,n\geq0)$ be an homogeneous Markov chain whose state space is the set of nonnegative integers (denoted by $\N$). Suppose that the transition matrix $P$ of $(\Lambda_n,n\geq0)$ is TP$_2$. Then, 
\begin{equation}\label{eq:TP2ThDTMC}
Q:\,(n,i)\longmapsto \Pb[\Lambda_n=i|\Lambda_0=0]=P^n(0,i)\text{ is TP}_2\text{ on }\N\times\N.
\end{equation}
\item[2)]Let $(\Lambda_t,t\geq0)$ be a continuous-time, homogeneous Markov chain whose state space is $\N$. Suppose that $(\Lambda_t,t\geq0)$ is right-continuous. If, for every $t\geq0$, the transition matrix $(P_t(i,j);\,i,j\in\N)$ is TP$_2$, then 
\begin{equation}\label{eq:TP2ThCTMC}
Q:\,(t,i)\longmapsto \Pb[\Lambda_t=i|\Lambda_0=0]=P_t(0,i)\text{ is TP}_2\text{ on }\R_+\times\N.
\end{equation}
\end{theorem}
\noindent 
The next result is an analog of Theorem \ref{theo:TP2MarkovChain} for $\R_+$-valued Markov processes and its proof is similar to that of Theorem \ref{theo:TP2MarkovChain}.  
\begin{theorem}\label{theo:TP2MarkovProcess}(Karlin \cite[Theorems 3.8(i) and 5.2]{KA}). 
\item[1)]Let $(\Lambda_n,n\in\N)$ be an $\R_+$-valued homogeneous Markov process. Suppose that its transition kernel is of the form $P(\theta,d\lambda)=p(\theta,\lambda)d\lambda$, where $p$ is continuous and TP$_2$ on $\R_+\times\R_+$. Let $(p^{(m)},m\in\N^{\ast})$ be the functions defined recursively as follows: $p^{(1)}=p$ and 
\begin{equation}\label{eq:TP2Recursive}
\forall\,m\geq2,\,\forall\,(\theta,\zeta)\in\R_+\times\R_+,\,p^{(m)}(\theta,\zeta)=\int_0^{\infty}p^{(m-1)}(\theta,\lambda)p(\lambda,\zeta)d\lambda.
\end{equation}
Then, 
\begin{equation}\label{eq:TP2ThDTMP}
q:\,(n,\lambda)\longmapsto p^{(n)}(0,\lambda)\text{ is TP}_2\text{ on }\N^{\ast}\times\R_+.
\end{equation}
\item[2)]Let $(\Lambda_t,t\geq0)$ be a right-continuous homogeneous Markov process whose state space is $\R_+$. Suppose that the transition kernel of $(\Lambda_t,t\geq0)$ is of the form $P_t(\theta,d\lambda)=p_t(\theta,\lambda)d\lambda$, where, for every $t>0$, $p_t$ is continuous and TP$_2$ on $\R_+\times\R_+$. Then,  
\begin{equation}\label{eq:TP2ThCTMP}
q:\,(t,\lambda)\longmapsto p_t(0,\lambda)\text{ is TP}_2\text{ on }\R_+^{\ast}\times\R_+,
\end{equation}
where $\R_+^{\ast}$ denotes the set of positive real numbers.
\end{theorem}
\noindent 
Theorems \ref{theo:TP2MarkovChain} and \ref{theo:TP2MarkovProcess} generally fail for inhomogeneous Markov processes. We refer to \cite[Section 2, pp. 50]{KA} for counterexamples. But, these results may be extended to random walks with $\R_+$-valued independent, non-stationary and log-concave increments.  
\begin{theorem}\label{theo:TP2RandWalk}(Karlin \cite[Theorem 3.7 (i)]{KA}).
Let $(\Lambda_n,n\in\N)$ be a random walk issued from $0$ with $\R_+$-valued independent  and log-concave increments (not necessary stationary). For every $n\in\N^{\ast}$, let $p_n$ denote the density of $\Lambda_n$. Then, 
\begin{equation}\label{eq:TPDeuxRandWalk}
q:(n,\lambda)\longmapsto p_n(\lambda)\text{ is TP}_2\text{ on }\N^{\ast}\times\R_+.
\end{equation}
\end{theorem}
\noindent
In the next section, we prove the equivalence between the MRL ordering and a log-concavity  type condition. This equivalence, together with Theorems \ref{theo:TP2MarkovChain}-\ref{theo:TP2RandWalk}, allow to construct new classes of processes which increase in the MRL order. 

\section{An equivalent condition to the MRL ordering}
Let $(X_t,t\geq0)$ be an integrable process. For every $t\geq0$, $\mu_t$ denotes the law of $X_t$, and $\Psi_{\mu_t}$ is the Hardy-Littlewood function of $\mu_t$:
\begin{equation}\label{eq:DefMRLFunct}
\Psi_{\mu_t}(x)=
\left\{
\begin{array}{ll}
x+\dfrac{1}{\mu_t([x,+\infty[)}\displaystyle\int_{[x,+\infty[}(y-x)\mu_t(dy)&\text{if }x<b_{\mu_t},\\&\\
x&\text{if }x\geq b_{\mu_t},
\end{array}
\right.
\end{equation}
with $$b_{\mu_t}=\inf\{z\in\R,\,\mu_t([z,+\infty[)=0\}.$$ We set $\mu=(\mu_t,t\geq0)$, and we denote by $C_{\mu}$ the function given by:
\begin{equation}\label{eq:CmudTail}
\forall\,(t,x)\in\R_+\times\R,\,C_{\mu}(t,x)=\E[(X_t-x)^+]=\int_{[x,+\infty[}(y-x)\mu_t(dy).
\end{equation}
We shall prove that the family $(\Psi_{\mu_t},t\geq0)$ is non-decreasing in $t$ if and only if $C_{\mu}$ is TP$_2$. For this end, we need some preliminary results. We begin by observing that $C_{\mu}$ determines entirely  the family  $\mu=(\mu_t,t\geq0)$. Indeed, the following result is proved in Hirsch-Roynette \cite[Section 2]{HR} (see also Müller-Stoyan \cite[Theorem 1.5.10]{MS}).
\begin{prop}\label{prop:DefCallFunct}
Let $\nu$ be an integrable probability measure and let $C_{\nu}$ denote the integrated survival function of $\nu$, i.e.  $C_{\nu}(x)=\int_{\R}(y-x)^{+}\nu(dy)$ for every $x\in\R$. 
Then $C_{\nu}$ enjoys the following properties:
\begin{enumerate}
\item[i)]$C_{\nu}$ is a convex, nonnegative function on $\R$,
\item[ii)]$\lim\limits_{x\to+\infty}C_{\nu}(x)=0$,
\item[iii)]there exists $l\in\R$ such that $\lim\limits_{x\to-\infty}(C_{\nu}(x)+x)=l$.  
\end{enumerate}
Conversely, if a function $C$ satisfies the above three properties, then there exists a unique integrable probability measure $\nu$ such that $C_{\nu}=C$, i.e. $C$ is the integrated survival function of $\nu$.
Precisely, $\nu$ is the second order derivative of $C$ in the sense of distributions, and $l=\int_{\R}y\nu(dy)$. 
\end{prop}
We also use the following result which states that the TP$_2$ property of $C_{\mu}$ is stronger than the increasing convex order.
\begin{prop}\label{prop:TP2vICO}
We suppose that the function $C_{\mu}$ given by (\ref{eq:CmudTail}) is TP$_2$. Then, for every $x\in\R$, $t\longmapsto C_{\mu}(t,x)$ is non-decreasing. Moreover, if $(X_t,t\geq0)$ has a constant mean, then $(X_t,t\geq0)$ is a peacock.  
\end{prop}
\begin{proof}
Let $0\leq s<t$ and $x\in\R$. We wish to show the inequality
\begin{equation}\label{eq:PcocI}
C_{\mu}(s,x)\leq C_{\mu}(t,x).
\end{equation}
If $C_{\mu}(s,x)=0$, then (\ref{eq:PcocI}) is obvious. Suppose that $C_{\mu}(s,x)>0$. Since $a\longmapsto C_{\mu}(s,a)$ is non-increasing, then, for every $z\leq x$, $C_{\mu}(s,z)>0$. Moreover, the TP$_2$ property of $C_{\mu}$ yields:
\begin{equation}\label{eq:TP2Ipcoc}
\forall\,z\leq x,\quad\frac{C_{\mu}(t,z)}{C_{\mu}(s,z)}\leq \frac{C_{\mu}(t,x)}{C_{\mu}(s,x)}.
\end{equation}
But, it follows from Point iii) in Proposition \ref{prop:DefCallFunct} that 
$$
\lim\limits_{z\to-\infty}\frac{C_{\mu}(t,z)}{z}=\lim\limits_{z\to-\infty}\frac{C_{\mu}(s,z)}{z}=-1.
$$
Thus, letting $z$ tend  to $-\infty$ in (\ref{eq:TP2Ipcoc}), we obtain
$$
1=\lim\limits_{z\to-\infty}\frac{C_{\mu}(t,z)}{C_{\mu}(s,z)}\leq\frac{C_{\mu}(t,x)}{C_{\mu}(s,x)}.
$$
If, in addition, $(X_t,t\geq0)$ has a constant mean, then we deduce from Proposition \ref{prop:PcocEquivdTail} that $(X_t,t\geq0)$ is a peacock.
\end{proof}
We now give an equivalent condition to the MRL ordering.
\begin{theorem}\label{theo:IMRVTP2Tail}
The process $(X_t,t\geq0)$ increases in the MRL order if and only if its integrated survival function $C_{\mu}$ is TP$_2$, i.e. for every $0\leq t_1\leq t_2$ and $x_1\leq x_2$,
\begin{equation}\label{eq:MRLTailTP2}
C_{\mu}\begin{pmatrix}
t_1,t_2\\
x_1,x_2
\end{pmatrix}=
\det
\begin{pmatrix}
C_{\mu}(t_1,x_1)&C_{\mu}(t_1,x_2)\\
C_{\mu}(t_2,x_1)&C_{\mu}(t_2,x_2)
\end{pmatrix}
\geq0.
\end{equation}
\end{theorem}
\begin{proof}
First, we recall that, for every $t\geq0$, $C_{\mu}(t,\cdot):x\longmapsto C_{\mu}(t,x)$ is continuous and 
left-differentiable, since it is convex. Moreover, if $C'_{\mu}(t,\cdot)$ denotes its left-derivative function, then, for every $x\in\R$, $C'_{\mu}(t,x)=-\mu_t([x,+\infty[)$.\\
Now, we suppose that $(X_t,t\geq0)$ increases in the MRL order, i.e. the family of functions $(\Psi_{\mu_v},v\geq0)$, defined by (\ref{eq:DefMRLFunct}), is non-decreasing in $v$. Let $0\leq s<t$ and $x\leq y$. We shall show that 
\begin{equation}\label{eq:TP2proof}
C_{\mu}(s,x)C_{\mu}(t,y)\geq C_{\mu}(t,x)C_{\mu}(s,y).
\end{equation}
If $C_{\mu}(s,y)=0$, then (\ref{eq:TP2proof}) is immediate. Suppose that $C_{\mu}(s,y)>0$. From \cite[Theorem 4.A.26]{ShS}, we know that the MRL ordering entails the increasing convex ordering. Therefore, $C_{\mu}(t,y)>0$.
Moreover, since $a\longmapsto C_{\mu}(s,a)$ and $a\longmapsto C_{\mu}(t,a)$ are non-increasing, we have:
\begin{equation}\label{eq:CmusCmutzleqy}
\forall\,a\leq y,\,C_{\mu}(s,a)>0\,\text{ and }\,C_{\mu}(t,a)>0.
\end{equation} 
Hence, (\ref{eq:TP2proof}) is equivalent to
\begin{equation}\label{eq:TP2proof2}
\frac{C_{\mu}(t,y)}{C_{\mu}(s,y)}\geq \frac{C_{\mu}(t,x)}{C_{\mu}(s,x)}.
\end{equation}
We also deduce from (\ref{eq:CmusCmutzleqy}) that, for every $a\leq y$, $\mu_s([a,+\infty[)>0$ and $\mu_t([a,+\infty[)>0$.\\
To obtain (\ref{eq:TP2proof2}), it suffices to show that 
$$
F_{s,t}:a\longmapsto\frac{C_{\mu}(t,a)}{C_{\mu}(s,a)}\,\text{ is non-decreasing on }]-\infty,y].
$$
Observe that $F_{s,t}$ is continuous and left-differentiable on $]-\infty,y]$, and if $F'_{s,t}$ denotes its left-derivative, then, for every $a\in]-\infty,y]$, we have:
\begin{align*}
&C_{\mu}^2(s,a)F'_{s,t}(a)\\
&=-C'_{\mu}(s,a)C_{\mu}(t,a)+C'_{\mu}(t,a)C_{\mu}(s,a)\\
&=\mu_s([a,+\infty[)C_{\mu}(t,a)-\mu_t([a,+\infty[)C_{\mu}(s,a)\\
&=\mu_s([a,+\infty[)\int_{[a,+\infty[}y\mu_t(dy)-\mu_t([a,+\infty[)\int_{[a,+\infty[}y\mu_s(dy)\\
&=\mu_{s}([a,+\infty[)\mu_t([a,+\infty[)\left(\Psi_{\mu_t}(a)-\Psi_{\mu_s}(a)\right)\geq0
\end{align*}
since, according to the hypothesis, $\Psi_{\mu_t}(a)\geq\Psi_{\mu_s}(a)$. Then, $F_{s,t}$ is non-decreasing and (\ref{eq:TP2proof2}) holds.\\
Conversely, suppose that $C_{\mu}$ is TP$_2$. By Proposition \ref{prop:TP2vICO}, $v\longmapsto C_{\mu}(v,x)$ is non-decreasing. Fix $x\in\R$ and $0\leq s\leq t$.  
If $x\geq b_{\mu_s}=\inf\{y\in\R,\,\mu_s([y,+\infty[)=0\}$, then, by definition of $\Psi_{\mu_s}$ and $\Psi_{\mu_t}$, 
$$
\Psi_{\mu_t}(x)\geq x=\Psi_{\mu_s}(x).
$$
Now, suppose that $x<b_{\mu_s}$. Then,  $C_{\mu}(s,x)>0$ and, as a consequence, $\mu_s([x,+\infty[)>0$. Moreover, since $v\longmapsto C_{\mu}(v,x)$ is non-decreasing, we also have $C_{\mu}(t,x)>0$ and $\mu_t([x,+\infty[)>0$. Hence, by the TP$_2$ property of  $C_{\mu}$,   
$$
F_{s,t}:a\longmapsto\frac{C_{\mu}(t,a)}{C_{\mu}(s,a)}
$$ 
is non-decreasing and left-differentiable on $]-\infty,x]$. In particular, the left-derivative of $F_{s,t}$ at $x$ is nonnegative, i.e.
\begin{align*}
0&\leq C^2(s,x)F'_{s,t}(x)\\
&=\mu_s([x,+\infty[)\int_{[x,+\infty[}y\mu_t(dy)-\mu_t([x,+\infty[)\int_{[x,+\infty[}y\mu_s(dy)  \\
&=\mu_{s}([x,+\infty[)\mu_t([x,+\infty[)\left(\Psi_{\mu_t}(x)-\Psi_{\mu_s}(x)\right)
\end{align*}
which shows that $t\longmapsto \Psi_{\mu_t}(x)$ is non-decreasing.
\end{proof} 
\begin{rem}
\item[1)]For certain processes, the equivalence between the MRL ordering and Condition (\ref{eq:MRLTailTP2}) (in Theorem \ref{theo:IMRVTP2Tail}) may be obtained by using a similar argument as in the proof of Theorem 1.8.4 in Müller-Stoyan \cite{MS}. Indeed, if $(X_t,t\geq0)$ is an integrable process such that each $X_t$ admits a positive density, then, for every $t\geq0$, the integrated survival function $C_t$ and the MRL function $L_{t}$ of $X_t$ are linked by the relation:
$$
\frac{d}{dx}\ln C_t(x)=-\frac{1}{L_t(x)}
$$
which yields that $(X_t,t\geq0)$ is MRL ordered if and only if Condition (\ref{eq:MRLTailTP2}) is fulfilled. Then, we obtain an alternative proof of the  direct implication part of Theorem \ref{theo:IMRVTP2Tail} by considering, for every $\varepsilon>0$, the process $(X_t^{\varepsilon},t\geq0)$ defined as
$$
X_t^{\varepsilon}=X_t+\varepsilon G,
$$
where $G$ is a reduced Gaussian random variable independent of $(X_t,t\geq0)$. Observe that each $X^{\varepsilon}_t$ admits a positive density, and if $(X_t,t\geq0)$ is MRL ordered, then, for every $\varepsilon>0$, $(X_t^{\varepsilon},t\geq0)$ is MRL ordered (see e.g. \cite[Lemma 2.A.8]{ShS}). Furthermore, if $C_t^{\varepsilon}$ denotes the integrated survival function  of $X_t^{\varepsilon}$, then, since $C_t$ is continuous, $(t,x)\longmapsto C_t(x)$ is TP$_2$ if and only if, for every $\varepsilon>0$, $(t,x)\longmapsto C^{\varepsilon}_t(x)$ is TP$_2$. 
\item[2)]In the case where $X_t=t X$, Condition (\ref{eq:MRLTailTP2}) rewrites
\begin{equation}\label{eq:MYTailTP2}
(t,x)\longmapsto tC\left(\frac{x}{t}\right)\text{ is TP}_2\text{ on }\R_+^{\ast}\times\R,
\end{equation}
where $C$ denotes the integrated survival function of $X$. If $\Psi$ is the Hardy-Littlewood function of $X$, then, by Theorem \ref{theo:IMRVTP2Tail}, the assertion (\ref{eq:MYTailTP2}) holds if and only if $(tX,t\geq0)$ is a MRL process, i.e.
\begin{equation}\label{eq:MYa}
\forall\,x\in\R,\,t\longmapsto t\Psi\left(\frac{x}{t}\right)\text{ is non-decreasing on }\R_+^{\ast}.
\end{equation}
Observe that, if $x\leq0$, then $t\longmapsto t\Psi(x/t)$ is non-decreasing since $\Psi$ is non-decreasing and nonnegative. Hence, (\ref{eq:MYa}) is equivalent to:
\begin{equation}\label{eq:MYb}
\forall\,x\in\R^{\ast}_+,\,t\longmapsto t\Psi\left(\frac{x}{t}\right)\text{ is non-decreasing on }\R_+^{\ast}
\end{equation}
which in turn is equivalent to the following Madan-Yor condition:
\begin{equation}\label{eq:MY} 
a\longmapsto\frac{\Psi(a)}{a}\text{ is non-increasing on }\R_+^{\ast}.
\end{equation}
Note that Condition (\ref{eq:MY}) means that $a\longmapsto\Psi(a)-\Psi(0)$ is {\it antistarshaped} at $0$ on $\R_+^{\ast}$. We refer to \cite[Page 237]{MS} for the definition of an antistarshaped function.\\
On the other hand, if we denote by $L$ the MRL function of $X$, then 
$$
\forall\,(t,x)\in\R^{\ast}_+\times\R,\text{ }t\Psi\left(\frac{x}{t}\right)=tL\left(\frac{x}{t}\right)+x
$$
and, as a consequence, (\ref{eq:MYb}) holds if and only if
\begin{equation}\label{eq:MYaMRL}
\forall\,x\in\R^{\ast}_+,\,t\longmapsto tL\left(\frac{x}{t}\right)\text{ is non-decreasing on }\R_+^{\ast}.
\end{equation} 
In particular, if $L$ is non-increasing, then (\ref{eq:MYaMRL}) holds. In other terms, if $X$ is a decreasing mean residual life  (DMRL) random variable, then $(tX,t\geq0)$ is a MRL process. This result extends Theorem 2.A.17 in Shaked-Shanthikumar \cite{ShS} to DMRL random variables whose supports are not contained in $\R_+$. 
\end{rem}
Theorem \ref{theo:IMRVTP2Tail} shall play an essential role in the last section where we exhibit new families of peacocks which increase in the MRL order.       
 
\section{Some examples of MRL processes}
There exists several classes of MRL processes. We start with time-dependent transformations of random variables. Such processes are exhibited by computing directly the corresponding Hardy-Littlewood functions.  Then, we apply Theorem \ref{theo:IMRVTP2Tail} and some closure properties of total positivity of order 2 to obtain other interesting families of MRL processes. In particular, we show that numerous subordinated MRL processes still increase in the MRL order. 
\subsection{Time-dependent transformations of random variables}
\subsubsection{Scale transformations}
If $Y$ is an integrable and centered random variable, then 
\begin{equation}\label{eq:ScaleMix1}
(\lambda Y,\lambda\geq0)\text{ is a centered MRL process}
\end{equation}
if and only if the Madan-Yor condition (\ref{eq:MY}) holds. Such examples, which orginated to Madan-Yor \cite{MY}, are also studied in Lim {\it et al. }\cite{LYY} and in Hirsch {\it et al. }\cite[Section 7.4]{HPRY}. These authors considered  the scale $\sqrt{\lambda}$ instead of $\lambda$. However, it is obvious that one may choose any scale of the form $h(\lambda)$, where $h$ is a nonnegative non-decreasing  function on $\R_+$.
There are many common random variables which satisfy (\ref{eq:MY}). For example, Gaussian, exponential, Gamma, Beta and Student random variables satisfy (\ref{eq:MY}).  There also exist random variables for which (\ref{eq:MY}) fails (see e.g. Hirsch {\it et al. }\cite[ Exercice 7.11]{HPRY}).
 
\subsubsection{Monotone transformations of IFR random variables}
Other examples of MRL processes are given in Hirsch {\it et al. }\cite[Section 7.4]{HPRY}. For instance, if $Y$ is a random variable which admits a positive continuous density and a log-concave survival function, then, assuming that $\E[e^{\lambda Y}]<\infty$ for every $\lambda\geq0$,
\begin{equation}\label{eq:ScaleMix2}
\left(\frac{e^{\lambda Y}}{\E[e^{\lambda Y}]}-1,\lambda\geq0\right)\text{ is a centered MRL process.}
\end{equation}
We shall exhibit two important classes of MRL processes. To present the first class, we introduce the set $\mathcal{H}$ of  functions
$\phi:\R_+\times\R\to\R$  such that: 
\begin{enumerate}
\item[(H1)]$\phi$ is continuous on $\R_+\times\R$, and of $\mathcal{C}^1$ class on $\R_+^{\ast}\times\R$,
\item[(H2)]for every $(\lambda,y)\in]0,+\infty[\times\R$,
\begin{equation}\label{eq:PartialPhiX}
 \frac{\partial \phi}{\partial y}(\lambda,y)>0,
\end{equation}
and for every $\lambda>0$,
\begin{equation}\label{eq:PartialPhiTX}
 y\,(\in\R)\longmapsto\widetilde{\phi}(\lambda,y):=\frac{\dfrac{\partial \phi}{\partial \lambda}(\lambda,y)}{\dfrac{\partial \phi}{\partial y}(\lambda,y)}\text{ is non-decreasing.}
\end{equation}
\item[(H3)]If, for every $\lambda>0$, $\tau_{+}(\lambda):=\lim\limits_{y\to+\infty}\phi(\lambda,y)$, resp. $\tau_{-}(\lambda)=\lim\limits_{y\to-\infty}\phi(\lambda,y)$ is the upper bound, resp. the lower bound of the interval $\phi(\lambda,\R)$, then
$$
\lambda\longmapsto\tau_{-}(\lambda)\text{ is non-increasing on }\R_+^{\ast}
$$
and either, for every $\lambda>0$, $\tau_{+}(\lambda)=+\infty$ or
$$
\lambda\longmapsto\tau_{+}(\lambda)\text{ is non-decreasing and of }\mathcal{C}^1\text{ class on }\R_+^{\ast}.
$$
\end{enumerate}
The following result, which generalizes (\ref{eq:ScaleMix2}), provides a large class of MRL peacocks.
\begin{theorem}\label{theo:PcocTailTP2a}
Let $\phi:\R_+\times\R\to\R$ belong to $\mathcal{H}$. 
Let $Y$ be a random variable such that
\begin{enumerate}
\item[i)]$Y$ admits a positive continuous density $f$, and 
\begin{equation}\label{eq:HypLogCncvTail}
\text{its survival function }m:z\longmapsto\Pb(Y\geq z)=\int_z^{+\infty}f(u)du\text{ is log-concave,}
\end{equation}
\item[ii)]for every $\eta>0$,
\begin{equation}\label{eq:PcocPhiSupInt}
\E[|\phi(\eta,Y)|]<\infty\,\text{ and }\,\E\left[\sup\limits_{0<\lambda\leq \eta}\left|\frac{\partial \phi}{\partial \lambda}(\lambda,Y)\right|\right]<\infty,
\end{equation}
\item[iii)]for every $\lambda\geq0$, $\E[\phi(\lambda,Y)]=0$.
\end{enumerate}
Then $(\phi(\lambda,Y),\lambda\geq0)$ is a centered MRL process.  
\end{theorem}
Here are some cases where Theorem \ref{theo:PcocTailTP2a} applies.
\begin{exa} Let $Y$ be a random variable satisfying Condition i) of Theorem \ref{theo:PcocTailTP2a} and let $\varphi:\R\to\R$ be a $\mathcal{C}^1$-class function such that $\varphi'>0$.
\item[1)]Suppose that:
$$
\forall\,\lambda>0,\,\text{ }\E\left[\sup\limits_{0<\eta\leq \lambda}|Y|\varphi'(\eta Y)\right]<\infty\,\text{  and  }\,\E[Y\varphi'(\lambda Y)]=0.
$$
Then 
$$
(\phi(\lambda,Y):=\varphi(\lambda Y)-\E[\varphi(\lambda Y)],\lambda\geq0)
$$
fulfills conditions of Theorem \ref{theo:PcocTailTP2a} with
$$
\tau_{-}(\lambda)=\tau_{-}=\lim\limits_{z\to-\infty}\varphi(z)-\varphi(0)\,\text{  and  }\,\tau_{+}(\lambda)=\tau_+=\lim\limits_{z\to+\infty}\varphi(z)-\varphi(0).
$$
\item[2)]We suppose that 
$$
\lim\limits_{y\to-\infty}\varphi(z)=l\,(l\in[-\infty,+\infty))\,\text{ and }\,\lim\limits_{z\to+\infty}\varphi(z)=+\infty.
$$
Observe that if
\begin{equation}\label{eq:exbPcocSupIntb}
\forall\,\lambda>0,\,\text{ }\E\left[\sup\limits_{0<\eta\leq \lambda}|Y|\varphi'(\eta Y)\right]<\infty,
\end{equation} 
then it follows from (\ref{eq:exbPcocSupIntb}) that $h:\lambda\longmapsto\E[\varphi(\lambda Y)]$ is differentiable on $\R_+^{\ast}$. We suppose moreover that $h$ is non-decreasing.
\begin{enumerate}
\item[i)]If $\varphi$ is convex, then
$$
y\longmapsto y-\frac{h'(\lambda)}{\varphi'(\lambda y)}\,\text{ is non-decreasing.}
$$
Therefore, Theorem \ref{theo:PcocTailTP2a} applies to
$$
(\phi(\lambda,Y):=\varphi(\lambda Y)-h(\lambda),\lambda\geq0)
$$
with
$$
\tau_{-}(\lambda)=l-h(\lambda)\,\text{ and }\,\tau_{+}=+\infty.
$$
In particular, if, for every $\lambda>0$, $\E[e^{\lambda Y}]<\infty$, and if $h:\lambda\longmapsto\E[e^{\lambda Y}]$ is non-decreasing, then
Theorem \ref{theo:PcocTailTP2a} applies to
$$
(\phi(\lambda,Y):=e^{\lambda Y}-h(\lambda),\lambda\geq0)
$$
with $\tau_{-}(\lambda)=-h(\lambda)$ and $\tau_{+}=+\infty$.\\
Note that, for any random variable $Z$ such that $\E[e^{\lambda Z}]<\infty$ for every $\lambda>0$, $(\phi(\lambda,Z):=e^{\lambda Z}-h(\lambda),\lambda\geq0)$ is a peacock if and only if $\lambda\longmapsto \E[e^{\lambda Z}]$ is non-decreasing (see \cite[Example 8]{BPR2}). 
\item[ii)]If $\varphi$ is log-convex, 
$$
y\longmapsto y-\frac{h'(\lambda)}{h(\lambda)}\frac{\varphi(\lambda y)}{\varphi'(\lambda y)}
\text{ is non-decreasing,}
$$
and, as a consequence, Theorem \ref{theo:PcocTailTP2a} applies to 
$$
\left(\phi(\lambda,Y):=\frac{\varphi(\lambda Y)}{h(\lambda)}-1,\lambda\geq0\right),
$$
with
$$
\tau_{-}(\lambda)=\frac{l}{h(\lambda)}-1\text{ }(\lambda>0)\text{ and }\tau_{+}(\lambda)=+\infty.
$$
In particular, if $\E\left[e^{\lambda Y}\right]<\infty$ for every $\lambda>0$, then
$$
\left(\phi(\lambda,Y):=\frac{e^{\lambda Y}}{\E\left[e^{\lambda Y}\right]}-1,\lambda\geq0\right)
$$
satisfies conditions of Theorem \ref{theo:PcocTailTP2a} with $\tau_{-}(\lambda)=\tau_{-}=-1$ and $\tau_{+}=+\infty$.
\end{enumerate}
\item[3)]Let $\varphi$ be concave and such that 
$$
\lim\limits_{y\to-\infty}\varphi(y)=-\infty\text{ and }\lim\limits_{y\to+\infty}\varphi(y)=l\text{ }(l\in (-\infty,+\infty])
$$
We suppose that $\varphi$ satisfies the following integrability hypothesis:
\begin{equation}\label{eq:ExaPcocTailTP2}
\forall\,\lambda\in\R_+,\text{ }\E[|\varphi(Y-\lambda)|]<+\infty\text{ and }\E[\varphi'(Y-\lambda)]\}<+\infty.
\end{equation}
We deduce from (\ref{eq:ExaPcocTailTP2}) that $h:\lambda\longmapsto\E[\varphi(Y-\lambda)]$ is a (non-increasing) $\mathcal{C}^1$-class function on $\R_+$. 
Since, for every $\lambda\in\R_+$,
$$
y\longmapsto-1-\frac{h'(\lambda)}{\varphi'(y-\lambda)}\text{ is non-decreasing,}
$$
then Theorem \ref{theo:PcocTailTP2a} applies to
$$
\left(\phi(\lambda,Y):=\varphi(Y-\lambda)-h(\lambda),\lambda\geq0\right),
$$
where $\tau_{-}(\lambda)=\tau_{-}=-\infty$ and $\tau_{+}(\lambda)=l-h(\lambda)$. 
 
\end{exa}
To prove Theorem \ref{theo:PcocTailTP2a}, we need  a convenient expression of the Hardy-Littlewood function.
\begin{lemma}\label{lem:PsimuExpress}
Let $\lambda>0$ be fixed. Let $\mu_{\lambda}$ be the law of $\phi(\lambda,Y)$, and let $\Psi_{\lambda}$ denote the Hardy-Littlewood function of $\mu_{\lambda}$. For every $z\in\R$, let $\phi^{-1}(\lambda,z)$ be the unique element in $\R$ such that
$$
\phi(\lambda,\phi^{-1}(\lambda,z))=z.
$$
Then, for every $y\in\R$, 
\begin{align*}
\Psi_{\lambda}(y)=&\left(y+\frac{1}{m(\phi^{-1}(\lambda,y))}\int_y^{\tau_{+}(\lambda)}m(\phi^{-1}(\lambda,z))dz\right)1_{]\tau_{-}(\lambda),\tau_{+}(\lambda)[}(y)
+y1_{[\tau_{+}(\lambda),+\infty[}(y),
\end{align*}
where $m$ denotes the survival function of $Y$.
\end{lemma}
\begin{proof}
Let $\lambda>0$ be fixed. Since 
$$
\int_{\tau_{-}(\lambda)}^{\tau_{+}(\lambda)}zd\mu_{\lambda}(z)=\E[\phi(\lambda,Y)]=0,
$$
then, for every  $y\in]-\infty,\tau_{-}(\lambda)]$, $\Psi_{\lambda}(y)=0$. Moreover,
if $y\in[\tau_{+}(\lambda),+\infty[$, then, by definition of $\Psi_{\lambda}$, $\Psi_{\lambda}(y)=y$. Suppose that $y\in]\tau_{-}(\lambda),\tau_{+}(\lambda)[$. Then, by definition of $\Psi_{\lambda}$,
$$
\Psi_{\lambda}(y)=y+\frac{1}{\mu_{\lambda}([y,\tau_{+}(\lambda)[)}\int_{[y,\tau_{+}(\lambda)[}(a-y)\mu_{\lambda}(da).
$$
But, 
$$
\mu_{\lambda}([y,\tau_{+}(\lambda)[)=\Pb(\phi(\lambda,Y)\geq y)=\Pb(Y\geq\phi^{-1}(\lambda,y))=m(\phi^{-1}(\lambda,y))
$$
and, by Tonelli's theorem,
$$
\int_{[y,\tau_{+}(\lambda)[}(a-y)\mu_{\lambda}(da)=\int_y^{\tau_{+}(\lambda)}\mu([z,\tau_{+}(\lambda)[)dz=\int_y^{\tau_{+}(\lambda)}m(\phi^{-1}(\lambda,z))dz.
$$
This completes the proof.
\end{proof}
\begin{proof}[Proof of Theorem \ref{theo:PcocTailTP2a}]
Let $0<\lambda_1<\lambda_2$ and $y\in\R$ be fixed. We shall prove that
\begin{equation}\label{eq:IMRVpExa}
\forall\,0<\lambda_1<\lambda_2,\,\forall\,y\in\R,\,\Psi_{\lambda_1}(y)\leq\Psi_{\lambda_2}(y).
\end{equation}
We deduce from (\ref{eq:PcocPhiSupInt}) and Hypothesis iii) that 
$$
0=\frac{\partial }{\partial \lambda}\E[\phi(\lambda,Y)]=\E\left[\frac{\partial \phi}{\partial \lambda}(\lambda,Y)\right]=\int_{-\infty}^{+\infty}\frac{\partial \phi}{\partial \lambda}(\lambda,a)f(a)da.
$$
Moreover, by Lemma \ref{lem:PsimuExpress},
\begin{align*}
\Psi_{\lambda}(y)=
\left(y+\frac{1}{m(\phi^{-1}(\lambda,y))}\int_y^{\tau_{+}(\lambda)}m(\phi^{-1}(\lambda,z))dz\right)1_{]\tau_{-}(\lambda),\tau_{+}(\lambda)[}(y)+y1_{[\tau_{+}(\lambda),+\infty[}(y).
\end{align*}
Since $\lambda\longmapsto\tau_{-}(\lambda)$ is non-increasing and $\lambda\longmapsto\tau_{+}(\lambda)$ is non-decreasing, we have $\tau_{-}(\lambda_2)\leq\tau_{-}(\lambda_1)$ and $\tau_{+}(\lambda_1)\leq\tau_{+}(\lambda_2)$.\\
Let us write  
$$
\R=]-\infty,\tau_{-}(\lambda_1)]\cup]\tau_{-}(\lambda_1),\tau_{+}(\lambda_1)[\cup[\tau_{+}(\lambda_1),+\infty[.
$$
If $y\in]-\infty,\tau_{-}(\lambda_1)]$, then $\Psi_{\lambda_1}(y)=0\leq\Psi_{\lambda_2}(y)$, and if $y\in[\tau_{+}(\lambda_1),+\infty[$, then $\Psi_{\lambda_1}(y)=y\leq\Psi_{\lambda_2}(y)$.   Now, Fix  $y\in]\tau_{-}(\lambda_1),\tau_{+}(\lambda_1)[$, and consider the function 
$$
\lambda\in(\lambda_1,\lambda_2)\longmapsto\Psi_{\lambda}(y)=y+\frac{1}{m(\phi^{-1}(\lambda,y))}\int_y^{\tau_{+}(\lambda)}m(\phi^{-1}(\lambda,z))dz,
$$
where, for every $z\in[y,+\infty[$, $\phi^{-1}(\lambda,z)$ denotes the unique element in $\R$ such that
\begin{equation}\label{eq:phiTXInverse}
\phi(\lambda,\phi^{-1}(\lambda,z))=z.
\end{equation}
Note that, since $\tau_{-}$, resp. $\tau_{+}$ is non-increasing, resp. non-decreasing, 
$$
\tau_{-}(\lambda)\leq\tau_{-}(\lambda_1)\leq y\leq\tau_{+}(\lambda_1)\leq\tau_{+}(\lambda).
$$
If we differentiate (\ref{eq:phiTXInverse}) with respect to $\lambda$, we obtain:
\begin{equation}\label{eq:phiDiffInverse}
\frac{\partial \phi^{-1}}{\partial \lambda}(\lambda,z)=-\frac{\dfrac{\partial \phi}{\partial \lambda}(\lambda,\phi^{-1}(\lambda,z))}{\dfrac{\partial \phi}{\partial y}(\lambda,\phi^{-1}(\lambda,z))}=-\widetilde{\phi}(\lambda,\phi^{-1}(\lambda,z)).
\end{equation}
Observe that, by Condition i), $m$ is of $\mathcal{C}^1$ class and denote by $m'$ the derivative of $m$.
Since $\tau_{+}$ is of $\mathcal{C}^1$ class on $\R_+^{\ast}$ and since $$m(\phi^{-1}(\lambda,\tau_{+}(\lambda)))=\Pb(\phi(\lambda,Y)\geq\tau_{+}(\lambda))=0,$$ the Leibniz' rule for differentiation under the integral sign yields that
\begin{align*}
&m^2(\phi^{-1}(\lambda,y))\frac{\partial \Psi_{\lambda}(y)}{\partial \lambda}\\
&=m^2(\phi^{-1}(\lambda,y))\frac{\partial }{\partial \lambda}
\left(y+\frac{1}{m(\phi^{-1}(\lambda,y))}\int_y^{\tau_{+}(\lambda)}m(\phi^{-1}(\lambda,z))dz\right)\\
&=m(\phi^{-1}(\lambda,y))\left(\int_y^{\tau_{+}(\lambda)}\frac{\partial \phi^{-1}}{\partial \lambda}(\lambda,z)m'(\phi^{-1}(\lambda,z))dz+m(\phi^{-1}(\lambda,\tau_{+}(\lambda)))\tau'_{+}(\lambda)\right)\\
&\qquad-\frac{\partial \phi^{-1}}{\partial \lambda}(\lambda,y)m'(\phi^{-1}(\lambda,y))\int_y^{\tau_{+}(\lambda)}m(\phi^{-1}(\lambda,z))dz\\
&=m(\phi^{-1}(\lambda,y))\int_y^{\tau_{+}(\lambda)}\widetilde{\phi}(\lambda,\phi^{-1}(\lambda,z))f(\phi^{-1}(\lambda,z))dz\\
&\qquad-f(\phi^{-1}(\lambda,y))\int_y^{\tau_{+}(\lambda)}\widetilde{\phi}(\lambda,\phi^{-1}(\lambda,y))m(\phi^{-1}(\lambda,z))dz 
\end{align*}
which is equivalent to:
\begin{align*}
&m^2(\alpha_{\lambda}(y))\frac{\partial \Psi_{\lambda}(y)}{\partial \lambda}\\
&=f(\alpha_{\lambda}(y))\int_y^{\tau_{+}(\lambda)}\left[\widetilde{\phi}(\lambda,\alpha_{\lambda}(z))-\widetilde{\phi}(\lambda,\alpha_{\lambda}(y))\right]m(\alpha_{\lambda}(z))dz\\
&\quad+f(\alpha_{\lambda}(y))\int_y^{\tau_{+}(\lambda)}\left(\dfrac{m(\alpha_{\lambda}(y))}{f(\alpha_{\lambda}(y))}-\dfrac{m(\alpha_{\lambda}(z))}{f(\alpha_{\lambda}(z))}\right)\widetilde{\phi}(\lambda,\alpha_{\lambda}(z))f(\alpha_{\lambda}(z))dz,
\end{align*}
where, for every $z\in[y,\tau_{+}(\lambda)[$, $\alpha_{\lambda}(z)=\phi^{-1}(\lambda,z)$.\\
We set
$$
K_1:= \int_y^{\tau_{+}(\lambda)}\left[\widetilde{\phi}(\lambda,\alpha_{\lambda}(z))-\widetilde{\phi}(\lambda,\alpha_{\lambda}(y))\right]m(\alpha_{\lambda}(z))dz
$$
and
$$
K_2:= \int_{\tau_{-}(\lambda)}^{\tau_{+}(\lambda)}1_{[y,\tau_{+}(\lambda)[}(z)\left(\dfrac{m(\alpha_{\lambda}(y))}{f(\alpha_{\lambda}(y))}-\dfrac{m(\alpha_{\lambda}(z))}{f(\alpha_{\lambda}(z))}\right)\widetilde{\phi}(\lambda,\alpha_{\lambda}(z))f(\alpha_{\lambda}(z))dz.
$$
Since $a\in]\tau_{-}(\lambda),\tau_{+}(\lambda)[\longmapsto\phi^{-1}(\lambda,a)$ is non-decreasing and since, by hypothesis (H2), $a\longmapsto\widetilde{\phi}(\lambda,a)$ is non-decreasing, we have $K_1\geq0$. Moreover, $m$ being log-concave, 
$$
\theta_y:z\in]\tau_{-}(\lambda),\tau_{+}(\lambda)[\longmapsto1_{[y,\tau_{+}(\lambda)[}(z)\left(\dfrac{m(\alpha_{\lambda}(y))}{f(\alpha_{\lambda}(y))}-\dfrac{m(\alpha_{\lambda}(z))}{f(\alpha_{\lambda}(z))}\right)
$$
is nonnegative and non-decreasing. On the other hand, $a\longmapsto\widetilde{\phi}(\lambda,a)$ is non-decreasing. Then, denoting by $\widetilde{\phi}^{-1}_{\lambda}$ the right-continuous inverse of $a\longmapsto\widetilde{\phi}(\lambda,a)$,
\begin{align*}
K_2&=\int_{\tau_{-}(\lambda)}^{\tau_{+}(\lambda)}\theta_y(z)\widetilde{\phi}(\lambda,\alpha_{\lambda}(z))f(\alpha_{\lambda}(z))dz\\
&=\int_{\tau_{-}(\lambda)}^{\tau_{+}(\lambda)}\theta_y(z)\widetilde{\phi}(\lambda,\phi^{-1}(\lambda,z))f(\phi^{-1}(\lambda,z))dz\\
&\geq\theta_y\left(\phi(\lambda,\widetilde{\phi}_{\lambda}^{-1}(0))\right) \int_{\tau_{-}(\lambda)}^{\tau_{+}(\lambda)}\widetilde{\phi}(\lambda,\phi^{-1}(\lambda,z))f(\phi^{-1}(\lambda,z))dz\\
&=\theta_y\left(\phi(\lambda,\widetilde{\phi}_{\lambda}^{-1}(0))\right) \int_{-\infty}^{+\infty}\frac{\partial \phi}{\partial \lambda}(\lambda,a)f(a)da=0\\
&\quad\text{(after the change of variable }a=\phi^{-1}(\lambda,z))
\end{align*}
which ends the proof of (\ref{eq:IMRVpExa}).
\end{proof}

\begin{rem}
\item[1)]Theorem \ref{theo:PcocTailTP2a} extends to functions which are positive on an interval $]l,r[$, and which vanish on $]l,r[^{c}$.
\item[2)]If the density $f$ is log-concave, then so is $$m:x\longmapsto\displaystyle\int_x^{+\infty}f(u)du,$$ and, using the same notations as in Theorem \ref{theo:PcocTailTP2a}, $(\phi(\lambda,Y),\lambda\geq0)$ is a peacock which increases in the MRL order. Indeed, it suffices to see that the function $(y,z)\longmapsto m(y-z)$ is TP$_2$. But, 
$$
m(y-z)=\int_{y}^{+\infty}f(a-z)\,da=\int_{-\infty}^{+\infty}1_{[y,+\infty[}(a)f(a-z)\,da,
$$
and the functions $(y,z)\longmapsto 1_{[y,+\infty[}(z)$ and $(y,z)\longmapsto f(y-z)$ are TP$_2$. Then the map $(y,z)\longmapsto m(y-z)$ is TP$_2$ as a convolution product of two TP$_2$ functions (see e.g. Proposition \ref{prop:TPCompForml}). 
\end{rem}  
The second family of MRL processes resembles the previous one. As in Theorem \ref{theo:PcocTailTP2a}, the random variables having a log-concave survival function play an essential role.
\begin{theorem}\label{theo:ExaIMRVprocess}
Let $\varphi:\R_+\to\R_+$ be a concave function such that $\varphi(0)=0$. Assume that $\varphi$ is of  $\mathcal{C}^1$ class and its derivative $\varphi'$ is positive. Let $Y$ be a random variable satisfying Condition i) of Theorem \ref{theo:PcocTailTP2a} such that $\E[\varphi(|Y|)]<\infty$. 
Then, for every non-decreasing $\mathcal{C}^1$-class function $g:\R_+\to\R$,
$$
\left(Y_{\lambda}:=\frac{\varphi((Y-g(\lambda))^+)}{\E[\varphi((Y-g(\lambda))^+)]},\lambda\geq0\right)\text{ is a MRL process.}
$$
\end{theorem}
Here are some examples where Theorem \ref{theo:ExaIMRVprocess} applies.
\begin{exa}
Let $Y$ be a random variable which enjoys Condition i) of Theorem \ref{theo:PcocTailTP2a}.  
\item[1)]If $Y$ is integrable, and if $g:\R_+\to\R$ is a non-decreasing $\mathcal{C}^1$-class function, then 
$$
\left(\frac{(Y-g(\lambda))^+}{\E[(Y-g(\lambda))^+]},\lambda\geq0\right)
$$
satisfies the conditions of Theorem \ref{theo:ExaIMRVprocess}.
\item[2)]For every non-decreasing $\mathcal{C}^1$-class function $g:\R_+\to\R$,
$$
\left(\frac{Arctan((Y-g(\lambda))^+)}{\E[Arctan((Y-g(\lambda))^+)]},\lambda\geq0\right)
$$
enjoys the conditions of Theorem \ref{theo:ExaIMRVprocess}.
\end{exa}
\begin{proof}[Proof of Theorem \ref{theo:ExaIMRVprocess}]
Let $h:\R_+\to\R_+$ be the function given by:
$$
\forall\,\lambda\geq0,\,h(\lambda)=\E[\varphi((Y-g(\lambda))^+)].
$$
Then, for every $\lambda\geq0$, 
\begin{equation}\label{eq:MeanIMRV0}
h(\lambda)=\int_0^{\tau}\Pb(\varphi((Y-g(\lambda))^+)\geq y)dy=\int_{0}^{\tau}m(\varphi^{-1}(y)+g(\lambda))dy,
\end{equation}
where $\tau=\lim\limits_{y\to+\infty}\varphi(y)\in]0,+\infty]$.
Observe that $h$ is a non-increasing $\mathcal{C}^1$-class function. Indeed, the tail $m$ of $Y$ is a non-increasing $\mathcal{C}^1$-class function and $g$ is a non-decreasing $\mathcal{C}^1$-class function. Moreover, we deduce from (\ref{eq:MeanIMRV0}) that
\begin{equation}\label{eq:MeanIMRV01}
1=\int_0^{\tau/h(\lambda)}m(\rho(\lambda,z))dz,
\end{equation}
with $\rho(\lambda,z)=\varphi^{-1}(h(\lambda)z)+g(\lambda)$. If we differentiate  (\ref{eq:MeanIMRV01}) with respect to $\lambda$, we obtain:
\begin{equation}\label{eq:MeanIMRV02}
0=\int_0^{\tau/h(\lambda)}\left[-\frac{\partial \rho}{\partial \lambda}(\lambda,z)\right]f(\rho(\lambda,z))dz,
\end{equation}
where, for every fixed $\lambda\in\R_+$, the function $\dfrac{\partial \rho}{\partial \lambda}(\lambda,\cdot)$ defined by
$$
\forall\,z\in(0,\tau/h(\lambda)),\text{ }\frac{\partial \rho}{\partial \lambda}(\lambda,z)=h'(\lambda)\frac{z}{\varphi'(\varphi^{-1}(h(\lambda)z))}+g'(\lambda)
$$
is non-increasing in $z$ for every fixed $\lambda$ (since $h'(\lambda)\leq0$ and $\varphi'$ is positive and non-increasing).\\
Now, for every $\lambda\geq0$, let $\Psi_{\lambda}$ denote the Hardy-Littlewood function attached to the law of $Y_{\lambda}$. It is not difficult to verify the following relation:
\begin{align*}
&\forall(\lambda,y)\in\R_+\times\R_+,\\
&\Psi_{\lambda}(y)=\left\{
\begin{array}{ll}
y&\text{if }y\geq\frac{\tau}{h(\lambda)},\\
y+\dfrac{1}{m(\rho(\lambda,y))}\displaystyle\int_y^{\tau/h(\lambda)}m(\rho(\lambda,z))dz&\text{if }0<y<\frac{\tau}{h(\lambda)},\\
1&\text{if }y=0.
\end{array}
\right.
\end{align*}
Let $0\leq\lambda_1<\lambda_2$ and $y\geq0$. We wish to prove that $\Psi_{\lambda_1}(y)\leq\Psi_{\lambda_2}(y)$. Note that, since $h$ is non-increasing, one has:
$$
\frac{\tau}{h(\lambda_1)}\leq\frac{\tau}{h(\lambda)}\leq\frac{\tau}{h(\lambda_2)}.
$$
We have $\Psi_{\lambda_1}(0)=1=\Psi_{\lambda_2}(0)$. If $y\geq\tau/h(\lambda_1)$, then $\Psi_{\lambda_1}(y)=y\leq\Psi_{\lambda_2}(y)$. Otherwise,  for every $(\lambda,y)\in(\lambda_1,\lambda_2)\times(0,\tau/h(\lambda_1))$,
\begin{align*}
&m^2(\rho(\lambda,y))\frac{\partial \Psi_{\lambda}(y)}{\partial \lambda}\\
&=m(\rho(\lambda,y))\int_y^{\tau/h(\lambda)}\left(\frac{\partial \rho}{\partial \lambda}(\lambda,y)-\frac{\partial \rho}{\partial \lambda}(\lambda,z)\right)f(\rho(\lambda,z))dz\\
&\quad+f(\rho(\lambda,y))\int_y^{\tau/h(\lambda)}\left(\frac{m(\rho(\lambda,y))}{f(\rho(\lambda,y))}-\frac{m(\rho(\lambda,z))}{f(\rho(\lambda,z))}\right)\left[-\frac{\partial \rho}{\partial \lambda}(\lambda,z)\right]f(\rho(\lambda,z))dz.
\end{align*} 
Therefore, we may conclude following the same lines as in the proof of Theorem \ref{theo:PcocTailTP2a}.
\end{proof}

\subsection{Subordination of MRL processes}
We shall deduce from Proposition \ref{prop:TPCompForml} and Theorem \ref{theo:IMRVTP2Tail}  that many subordinated MRL processes are MRL processes too. Here is our main result.
\begin{theorem}\label{theo:IMRVProcess}
Let $(Y_{\lambda},\lambda\geq0)$ be a centered MRL process. Consider  an
$\R_+$-valued Markov process $\Lambda:=(\Lambda_t,t\geq0)$ which satisfies the following conditions:
\begin{enumerate}
\item[i)]$\Lambda$ is homogeneous, right-continuous, and started at $0$.
\item[ii)]The transition kernel of $\Lambda$ has the form $P_t(\theta,d\lambda)=p_t(\theta,\lambda)d\lambda$, where, for every $t\geq0$, $p_t$ is continuous and TP$_2$.
\item[iii)]$\Lambda$ is independent of $(Y_{\lambda},\lambda\geq0)$, and, for every $t\geq0$, $\E\left[\left|Y_{\Lambda_t}\right|\right]<\infty$.
\end{enumerate}
Then, the subordinated process $(X_t:=Y_{\Lambda_t},t\geq0)$ is still a centered MRL process.  
\end{theorem}
\begin{rem}
Provided that the total positivity assumption holds, Theorem \ref{theo:IMRVProcess} still applies in the following cases:
\begin{itemize}
\item $(\Lambda_t,t\geq0)$ is an homogeneous (continuous-time) Markov chain whose state space is $\N$,
\item $(\Lambda_t,t\geq0)$ is a random walk with $\R_+$-valued independent increments.
\end{itemize}
\end{rem}
\begin{proof}[Proof of Theorem \ref{theo:IMRVProcess}]
Note that $(X_t,t\geq0)$ is obviously centered. Then, thanks to Proposition \ref{prop:TP2vICO} and Theorem \ref{theo:IMRVTP2Tail}, it suffices to show that the integrated survival function $C^{X}$ of $(X_t,t\geq0)$ is TP$_2$ on $\R_+^{\ast}\times\R$. Let $C^{Y}$ denote the integrated survival function of $(Y_{\lambda},\lambda\geq0)$. Since $(Y_{\lambda},\lambda\geq0)$ is a MRL process, we deduce from Theorem \ref{theo:IMRVTP2Tail} that $C^{Y}$ is TP$_2$ on $\R_+^{\ast}\times\R$. On the other hand, by Theorem \ref{theo:TP2MarkovProcess} (Point 2)),
$$
q:\,(t,\lambda)\longmapsto p_t(0,\lambda)\text{ is TP}_2\text{ on }\R_+^{\ast}\times\R.
$$
Moreover, for every $(t,x)\in\R_+^{\ast}\times\R$, 
\begin{align*}
C^{X}(t,x)&=\E_0\left[(X_t-x)^{+}\right]=\E_0\left[\left(Y_{\Lambda_t}-x\right)^{+}\right]\\
&=\int_0^{+\infty}\E\left[(Y_{\lambda}-x)^{+}\right]p_t(0,\lambda)d\lambda\\
&=\int_0^{+\infty}q(t,\lambda)C^{Y}(\lambda,x)d\lambda.
\end{align*}
Since $q$ and $C^{Y}$ are TP$_2$, then, we deduce from Proposition \ref{prop:TPCompForml}  that $C^{X}$ is also TP$_2$. This ends the proof.
\end{proof}
\begin{rem}\label{rem:IMRVprocessPlus}
Let $\Lambda:=(\Lambda_t,t\geq0)$ be an $\R_+$-valued Markov process satisfying Conditions i) and ii) of Theorem \ref{theo:IMRVProcess}. Then, $\Lambda$ increases in both stochastic and MRL orders. To see that $\Lambda$ increases stochastically, one may notice that, since $(t,\eta)\longmapsto p_t(0,\eta)$ and $(\lambda,\eta)\longmapsto1_{[\lambda,+\infty[}(\eta)$ are TP$_2$ on $\R_+^{\ast}\times\R_+$, Proposition \ref{prop:TPCompForml} implies:
\begin{equation}\label{eq:TP2TailIncrease}
(t,\lambda)\longmapsto\Pb_0(\Lambda_t\geq\lambda)=\int_{\R_+}1_{[\lambda,+\infty[}(\eta)p_t(0,\eta)d\eta\text{ is TP}_2\text{ on }\R^{\ast}_+\times\R_+
\end{equation}
which in turn implies that, for every $0< s\leq t$,
$$
\lambda\,(\in\R_+)\longmapsto\frac{\Pb_0(\Lambda_t\geq\lambda)}{\Pb_0(\Lambda_s\geq\lambda)}
$$
(taking the value $1$ at $\lambda=0$)
is increasing on the interval $\{\lambda\in\R_+:\,\Pb(\Lambda_s\geq\lambda)>0\}$. Therefore, a result obtained in Hirsch {\it et al. }\cite[Exercise 1.26]{HPRY} yields that, for every peacock $(Y_{\lambda},\lambda\geq0)$ satisfying Condition iii) of Theorem \ref{theo:IMRVProcess}, $(Y_{\Lambda_t},t\geq0)$ is a peacock.\\ The assertion  (\ref{eq:TP2TailIncrease}) implies also that 
$$
(t,\lambda)\longmapsto\E_0\left[(\Lambda_t-\lambda)^{+}\right]\text{ is TP}_2\text{ on }\R^{\ast}_+\times\R_+
$$
which means that $\Lambda$ is a MRL process.
\end{rem}
We now give some examples of MRL processes obtained by subordination.
\subsubsection{Scale mixtures of random variables}
The next result follows immediately from Assertion (\ref{eq:ScaleMix1}) and from Theorem \ref{theo:IMRVProcess}.
\begin{corol}\label{corol:IMRVprocess}
Let $(\Lambda_t,t\geq0)$ be an $\R_+$-valued Markov process which fulfills Conditions i) and ii) of Theorem \ref{theo:IMRVProcess}.
Let $Y$ be an integrable random variable independent of $(\Lambda_t,t\geq0)$.  
Suppose that $Y$ satisfies the Madan-Yor condition (\ref{eq:MY}). If $\E[Y]=0$ and if, for every $t\geq0$, $\E[\Lambda_t]<\infty$,
then, 
\begin{equation}\label{eq:ExaIMRV01}
(X_t:=Y\Lambda_t,t\geq0)\text{ is a centered MRL process.}
\end{equation}
\end{corol}
\begin{rem}
The assertion (\ref{eq:ExaIMRV01}) means that $(X_t,t\geq0)$ is a peacock and, if $\mu_t$ denotes the law of $X_t$, if $T_{\mu_t}$ denotes the Azéma-Yor stopping time for $\mu_t$, and if $(B_t,t\geq0)$ is a Brownian motion issued from $0$, then $\left(B_{T_{\mu_t}},t\geq0\right)$ is a martingale associated to $(X_t,t\geq0)$.
\end{rem} 
\begin{exa}
Let $Y$ denote an integrable and centered random variable which satisfies the Madan-Yor condition (\ref{eq:MY}). Let $\Lambda:=(\Lambda_t,t\geq0)$ be an integrable,  $\R_+$-valued homogeneous Markov process issued from $0$ and independent of $Y$. Then, in each situation below, $(Y\Lambda_t,t\geq0)$ is a centered MRL process:
\begin{itemize}
\item $\Lambda$ is right-continuous and has independent log-concave increments.
\item $\Lambda$ is a birth-death process.
\item $\Lambda$ is a diffusion which admits a continuous transition density. 
\end{itemize}
\end{exa}
\begin{exa} 
\item[1)] Let $(\Theta_n,n\in\N^{\ast})$ be independent exponential random variables of the same parameter $c$ ($c>0$).
We know that each $\Theta_n$ admits the density  $p:\lambda\longmapsto ce^{-c\lambda}1_{\R_+}(\lambda)$ and that the random walk $(\Lambda_n,n\in\N)$ given by $\Lambda_0=0$ and
\begin{equation}\label{eq:ExpRwalk}
\forall\,n\in\N^{\ast},\,\Lambda_n=\Lambda_{n-1}+\Theta_n
\end{equation} 
is an $\R_+$-valued homogeneous Markov process issued from $0$ with transition kernel $P(\lambda,d\gamma)=p(\lambda-\xi)d\xi$. For every $n\in\N^{\ast}$, $\Lambda_n$ follows the Erlang distribution with parameters $c$ and $n$. In other terms, the density of $\Lambda_n$ is the $n$-fold convolution $p^{(n)}$ of $p$ given by: 
$$
\forall\,n\in\N^{\ast},\,\forall\,\lambda\in\R,\,p^{(n)}(\lambda)=\frac{c^n}{\Gamma(n)}\lambda^{n-1}e^{-c\lambda }1_{\R_+}(\lambda).
$$ 
Since $p$ is a PF$_2$ function, then, by Point 1) of Theorem \ref{theo:TP2MarkovProcess}, 
$$
(n,\lambda)\longmapsto p^{(n)}(\lambda)\text{ is TP}_2\text{ on }\N^{\ast}\times\R_+.
$$
If $Y$ is an integrable and centered random variable independent of $(\Lambda_n,n\in\N)$ which satisfies Condition (\ref{eq:MY}), then, by Theorem \ref{theo:IMRVProcess}, 
$$
(X_n=Y\Lambda_n,n\in\N)\text{ is a centered MRL process.}
$$
\item[2)]We now consider a family $(\Theta_n,n\in\N^{\ast})$ of independent random variables such that each $\Theta_n$ is an exponential random variable of parameter $c_n$ ($c_n>0$), i.e. each $\Theta_n$ has the density $f_n:\lambda\longmapsto c_ne^{-c_n\lambda}1_{\R_+}(\lambda)$. The $\R_+$-valued random walk $(\Lambda_n,n\in\N)$ defined by $\Lambda_0=0$ and (\ref{eq:ExpRwalk}) has independent and log-concave increments since each $f_n$ is a PF$_2$ function. For every $n\in\N^{\ast}$, we denote by $p_n$ the $n$-fold convolution $f_1\ast\cdots\ast f_n$. By Theorem \ref{theo:TP2RandWalk}, 
\begin{equation}\label{eq:TP2ExpRwalk}
(n,\lambda)\longmapsto p_n(\lambda)\text{ is TP}_2\text{ on }\N^{\ast}\times\R_+.
\end{equation}
Let $Y$ be an integrable and centered  random variable, independent of $(\Lambda_n,n\in\N)$ which satisfies Condition (\ref{eq:MY}). Then we deduce from (\ref{eq:TP2ExpRwalk}) and Theorem \ref{theo:IMRVProcess} that
$$
(X_n=Y\Lambda_n,n\in\N)\text{ is a centered MRL process.}
$$
\item[3)]Let $(W_n,n\in\N)$ be the  random walk issued from $0$ whose increments admit the same density $p:\lambda\longmapsto (c/2)e^{-c|\lambda|}$, where $c>0$. The absolute value process $(\Lambda_n=|W_n|,n\in\N)$   is an $\R_+$-valued homogeneous Markov process with transition kernel $P(\lambda,d\gamma)={\bf p}(\lambda,\gamma)d\gamma$, where ${\bf p}$ is defined by:
$$
\forall\,(\lambda,\gamma)\in\R_+\times\R_+,\,{\bf p}(\lambda,\gamma)=p(-\lambda-\gamma)+p(-\lambda+\gamma).
$$
We mention that, for every $\lambda\geq0$ and every $n\in\N$, ${\bf p}(\lambda,\cdot)$ is the density of the random variable $|\lambda+W_{n+1}-W_n|$. Moreover, 
$$
\forall\,a\in]-c,c[,\text{ }\frac{1}{\E\left[e^{aW_1}\right]}=1-\frac{a^2}{c^2}.
$$
By Schoenberg's characterization of PF$_{\infty}$ functions  (see \cite{Sch}, Theorem 1), $p$ has the PF$_{\infty}$ property. Then, we deduce from Theorem 11.1 of Karlin \cite{KA} that ${\bf p}$ is TP$_2$ on $\R_+\times\R_+$. Let $p^{(n)}$ be the $n$-fold convolution of $p$ and, for every $\lambda\in\R_+$, let ${\bf p}^{(n)}(\lambda,\cdot)$ denote the density of $|\lambda+W_n|$. We have:
$$
\forall\,\gamma\in\R_+,\,{\bf p}^{(n)}(\lambda,\gamma)=p^{(n)}(-\lambda-\gamma)+p^{(n)}(-\lambda+\gamma).
$$
Applying Point 1) of Theorem \ref{theo:TP2MarkovProcess},
$$
(n,\lambda)\longmapsto{\bf p}^{(n)}(0,\lambda)\text{ is TP}_2\text{ on }\N^{\ast}\times\R_+.
$$
Now, let $Y$ be an integrable and centered random variable independent of $(\Lambda_n,n\in\N)$ which satisfies Condition (\ref{eq:MY}). Then, Condition (\ref{eq:TP2ExpRwalk}) and Theorem \ref{theo:IMRVProcess} yield that
$$
\left(X_n=Y\Lambda_n,n\in\N\right)\text{ is a centered MRL process.}
$$
\end{exa}
\subsubsection{Mixtures of IFR random variables}
As a consequence of Theorem \ref{theo:IMRVProcess}, Theorem \ref{theo:PcocTailTP2a} and Theorem \ref{theo:ExaIMRVprocess}, we have:
\begin{corol}
Consider  an $\R_+$-valued Markov process $(\Lambda_t,t\geq0)$ which fulfills Conditions i) and ii) of Theorem \ref{theo:IMRVProcess} and consider a random variable $Y$ independent of $(\Lambda_t,t\geq0)$.
\item[1)] Suppose that $Y$ satisfies Hypotheses i), ii) and iii) of Theorem \ref{theo:PcocTailTP2a}. If $\phi:\R_+\times\R\to\R$ is an element of $\mathcal{H}$ such that, for every $t\geq0$, $\E[\phi(\Lambda_t,Y)]<\infty$, then 
$$
(X_t:=\phi(\Lambda_t,Y),t\geq0)\text{ is a centered MRL process.}
$$

\item[2)] Suppose that $Y$ satisfies Condition i) of Theorem \ref{theo:PcocTailTP2a} and that $\E[\varphi(|Y|)]<\infty$. If $\varphi:\R_+\to\R_+$ is a concave $\mathcal{C}^1$-class function which satisfies $\varphi'>0$ and $\varphi(0)=0$ and if $g:\R_+\to\R$ denotes a non-decreasing $\mathcal{C}^1$-class function such that, for every $t\geq0$, $\E[\varphi(|g(\Lambda_t)|)]<\infty$, then 
$$
\left(X_t:=\frac{\varphi\left((Y-g(\Lambda_t))^+\right)}{h(\Lambda_t)}-1,t\geq0\right)\text{ is a centered MRL process,}
$$
where, for every $\lambda\in\R_+$, $h(\lambda)=\E\left[\varphi\left((Y-g(\lambda))^+\right)\right]$.  
\end{corol}
\begin{exa}
Let $(\Lambda_t,t\geq0)$ be an  $\R_+$-valued Markov process satisfying Conditions i) and ii) of Theorem \ref{theo:IMRVProcess}. Let $Y$ be a random variable independent of $(\Lambda_t,t\geq0)$ which fulfills Condition i) of Theorem \ref{theo:PcocTailTP2a}.
\item[1)]If, for every $\lambda\geq0$, $\E[e^{\lambda Y}]<\infty$, then
\begin{equation*}
\left(X_t:=\frac{e^{Y\Lambda_t}}{h(\Lambda_t)}-1,t\geq0\right)\text{ is a centered MRL process,}
\end{equation*}
where, for every $\lambda\geq0$, $h(\lambda)=\E[e^{\lambda Y}]$. 
\item[2)]If $Y$ is integrable, then
\begin{equation*}
\left(X_t:=\frac{(Y-\Lambda_t)^{+}}{h(\Lambda_t)}-1,t\geq0\right)\text{ is a centered MRL process,}
\end{equation*}
where, for every $\lambda\geq0$, $h(\lambda)=\E[(Y-\lambda)^{+}]$.
\end{exa}

\subsection{Other closure properties of the MRL ordering}

\subsubsection{Translation}
Let $(X_t,t\geq0)$ be a MRL process and let  
$Y$ be an integrable random variable independent of $(X_t,t\geq0)$. If $Y$ has a log-concave survival function, then we deduce from \cite[Lemma 2.A.8]{ShS} that $(Z_t:=X_t+Y,t\geq0)$ is still a MRL process. 

\subsubsection{Scale mixtures}
Let  $(X_t,t\geq0)$ be a MRL process and let
 $Y$ be an integrable $\R_+$-valued random variable which admits a positive $\mathcal{C}^1$-class density $f$. We suppose that $\log Y$ is log-concave, i.e. the density of $Y$ fulfills the following equivalent conditions:
\begin{itemize}
\item $y\longmapsto yf'(y)/f(y)$ is non-decreasing.
\item $f=e^{-V}$, where $y\longmapsto yV'(y)$ is a non-decreasing function.
\item For every $c\in]0,1[$, 
\begin{equation}\label{eq:loglogY0}
y\longmapsto\frac{f(y)}{f(yc)}\text{ is non-decreasing.}
\end{equation}
\end{itemize}
Moreover, it is obvious that (\ref{eq:loglogY0}) is equivalent to
\begin{equation}\label{eq:loglogY1}
(x,y)\longmapsto f\left(\frac{y}{z}\right)\text{ is TP}_2\text{ on }\R_+\times\R_+^{\ast}.
\end{equation} 
Let $(X_t,t\geq0)$ be a MRL process, and, for every $t\geq0$, let $Z_t=YX_t$. We denote by $C^{X}$, resp. $C^{Z}$ the integrated survival function of $(X_t,t\geq0)$, resp. $(Z_t,t\geq0)$. For every $(t,x)\in\R^{\ast}_+\times\R_+^{\ast}$, we have:
\begin{align*}
C^{Z}(t,x)&=\E[(YX_t-x)^+]=\int_0^{+\infty}C^{X}\left(t,\frac{x}{y}\right)yf(y)dy\\
&=\int_0^{+\infty}C^{X}(t,z)\frac{x^2}{z^3}f\left(\frac{x}{z}\right)dz\,\left(\text{by the change of variable }z=\frac{x}{y}\right).
\end{align*}
Since $C^{X}$ is TP$_2$ on $\R_+^{\ast}\times\R_+^{\ast}$, we deduce from (\ref{eq:loglogY1}) and from Proposition \ref{prop:TPCompForml} that $C^{Z}$ is TP$_2$ on $\R_+^{\ast}\times\R_+^{\ast}$. Similarly, for every $(t,x)\in\R_+^{\ast}\times\R_{-}^{\ast}$,
$$
C^{Z}(t,x)=-\int_{-\infty}^{0}C^{X}(t,z)\frac{x^2}{z^3}f\left(\frac{-x}{-z}\right)dz
$$ 
which shows that $C^{Z}$ is still TP$_2$ on $\R_+^{\ast}\times\R_{-}^{\ast}$. Thanks to the continuity property of $C^{Z}$, we deduce that $C^{Z}$ is TP$_2$ on $\R_+^{\ast}\times\R$ which means that $(Z_t:=YX_t,t\geq0)$ is increasing in the MRL order.

\subsubsection{Mixtures of MRL ordered probability measures}
Let $(\mu_n,n\in\N)$ be an integrable family of probability measures which increases in the MRL order. Let $c=(c_n,n\in\N)$ be an increasing sequence of nonnegative real numbers. We set $l=\lim\limits_{n\to\infty}c_n$. Then, the family $(\mu^{(c)}_t,t\in[0,l))$ defined by: 
$$
\forall\,t\in[c_n,c_{n+1}],\,\mu^{(c)}_t=\frac{c_{n+1}-t}{c_{n+1}-c_n}\mu_{n}+\frac{t-c_n}{c_{n+1}-c_{n}}\mu_{n+1}
$$
is still MRL ordered. This is a direct consequence of Theorem 2.A.18 in \cite{ShS}.
 
\subsubsection{Censoring type transformations }
Let $(\mu_t,t\geq0)$ be a family of integrable and centered probability measures which increases in the MRL order. By Theorem \ref{theo:IMRVTP2Tail}, the MRL ordering of $(\mu_t,t\geq0)$ rewrites:
\begin{equation}
C:(t,x)\longmapsto\int_{[x,+\infty[}(y-x)\mu_t(dy)\text{ is TP}_2\text{ on }\R_+\times\R.
\end{equation}
Let $a<b$ be fixed real numbers. We consider the family $(\mu^{a,b}_t,t\geq0)$ given by:
\begin{equation}\label{eq:CensorTr}
\mu^{a,b}_t(dy)=(1_{]-\infty,a[}+1_{]b,+\infty[})(y)\mu_t(dy)+\alpha^{a,b}_t\delta_a(dy)+\beta^{a,b}_t\delta_b(dy),
\end{equation}
where $\delta_a$, resp. $\delta_b$ denotes the Dirac measure at point $a$, resp. point $b$, and where
$$
\alpha^{a,b}_t=\frac{1}{b-a}\int_{[a,b]}(b-y)\mu_t(dy)\,\text{ and }\,\beta^{a,b}_t=\frac{1}{b-a}\int_{[a,b]}(y-a)\mu_t(dy).
$$
Observe that $\alpha^{a,b}_t$ and $\beta^{a,b}_t$ are chosen in such a way that $\mu^{a,b}_t$ is still a centered probability measure. Precisely, $(\alpha^{a,b}_t,\beta^{a,b}_t)$ is the unique solution of the linear system:
$$
\alpha+\beta=\mu_t([a,b])\,\text{ and }\,a\alpha+b\beta=\int_{[a,b]}y\mu_t(dy).
$$
Let $C^{a,b}:\R_+\times\R\to\R_+$ be the integrated survival function of $(\mu^{a,b}_t,t\geq0)$:
$$
C^{a,b}(t,x)=\int_{[x,+\infty[}(y-x)\mu^{a,b}_t(dy).
$$
We wish to prove that $(\mu_t^{a,b},t\geq0)$ is a MRL family. To this end, it is necessary and sufficient to show that $C^{a,b}$ is TP$_2$ on $\R_+\times\R$. For every $t\geq0$, we have:
$$
C^{a,b}(t,x)=\left\{
\begin{array}{ll}
C(t,x)&\text{if }x> b,\\
(b-x)\beta^{a,b}_t+\displaystyle\int_{]b,+\infty[}(y-x)\mu_t(dy)&\text{if }a\leq x\leq b,\\
\displaystyle\int_{[x,a[\cup]b,+\infty[}(y-x)\mu_t(dy)+(a-x)\alpha^{a,b}_t+(b-x)\beta^{a,b}_t&\text{if }x<a 
\end{array}
\right.
$$
which is equivalent to
$$
C^{a,b}(t,x)=\left\{
\begin{array}{ll}
\dfrac{b-x}{b-a}C(t,a)+\dfrac{x-a}{b-a}C(t,b)&\text{if }x\in[a,b]\\&\\
C(t,x)&\text{otherwise.} 
\end{array}
\right.
$$
Let $0\leq t_1\leq t_2$ and $x_1\leq x_2$ be fixed real numbers. We distinguish four cases:\\
i) If $x_1\notin [a,b]$ and $x_2\notin[a,b]$, 
\begin{align*}
C^{a,b}\begin{pmatrix}
t_1,t_2\\x_1,x_2
\end{pmatrix}&=C^{a,b}(t_1,x_1)C^{a,b}(t_2,x_2)-C^{a,b}(t_1,x_2)C^{a,b}(t_2,x_1)\\
&=C(t_1,x_1)C(t_2,x_2)-C(t_1,x_2)C(t_2,x_1)\\
&=C\begin{pmatrix}
t_1,t_2\\x_1,x_2
\end{pmatrix}\geq0\text{ (since }C\text{ is TP}_2.)
\end{align*}
ii) If $x_1<a\leq x_2\leq b$,
\begin{align*}
C^{a,b}\begin{pmatrix}
t_1,t_2\\x_1,x_2
\end{pmatrix}=\frac{b-x_2}{b-a}C\begin{pmatrix} 
t_1,t_2\\x_1,a 
\end{pmatrix}+\frac{x_2-a}{b-a}C\begin{pmatrix}
t_1,t_2\\x_1,b
\end{pmatrix}\geq0.
\end{align*}
iii) If $a\leq x_1\leq b\leq x_2$,
\begin{equation*}
C^{a,b}\begin{pmatrix}
t_1,t_2\\x_1,x_2
\end{pmatrix}=\frac{b-x_1}{b-a}C\begin{pmatrix}
t_1,t_2\\a,x_2
\end{pmatrix}+\frac{x_1-a}{b-a}C\begin{pmatrix}
t_1,t_2\\b,x_2
\end{pmatrix}\geq0.
\end{equation*}
iv) If $a\leq x_1\leq x_2\leq b$,
\begin{equation*}
C^{a,b}\begin{pmatrix}
t_1,t_2\\x_1,x_2
\end{pmatrix}=\frac{x_2-x_1}{b-a}C\begin{pmatrix}
t_1,t_2\\a,b
\end{pmatrix}\geq0.
\end{equation*}
We deduce from i)-iv) that $C^{a,b}$ is TP$_2$ on $\R_+\times\R$ which, according to Theorem \ref{theo:IMRVTP2Tail}, means that $(\mu^{a,b}_t,t\geq0)$ is a MRL family.\\
The transformation (\ref{eq:CensorTr}) may be generalized as follows. Let us consider a positive integer $k$, $k+1$ real numbers $a_0<a_1<\cdots<a_k$ and the $k+1$-tuple ${\bf a}=(a_0,a_1,\cdots,a_k)$. Let $(\mu^{\bf a}_t,t\geq0)$ be the family defined as
\begin{equation}\label{eq:RcensorTr}
\mu^{\bf a}_t(dy)=\left(1_{]-\infty,a_0[}+1_{]a_k,+\infty[}\right)(y)\mu_t(dy)+\sum\limits_{n=0}^k\alpha_t^{a_n}\delta_{a_n}(dy),
\end{equation}
where $\delta_{a_n}$ is the Dirac measure at point $a_n$ and where $\alpha_t^{a_0}, \alpha_t^{a_1},\cdots,\alpha_t^{a_k}$ are given by:
\begin{equation*}
\alpha_t^{a_n}=\left\{
\begin{array}{ll}
\dfrac{1}{a_1-a_0}\displaystyle\int_{[a_0,a_1]}(a_1-y)\mu_t(dy)&\text{if }n=0,\\&\\
\dfrac{1}{a_n-a_{n-1}}\displaystyle\int_{[a_{n-1},a_n]}(y-a_{n-1})\mu_t(dy)+ &\\
\qquad\qquad\qquad\quad\dfrac{1}{a_{n+1}-a_n}\displaystyle\int_{]a_n,a_{n+1}]}(a_{n+1}-y)\mu_t(dy)&\text{if }n=1,\cdots,k-1,\\&\\
\dfrac{1}{a_k-a_{k-1}}\displaystyle\int_{[a_{k-1},a_k]}(y-a_{k-1})\mu_t(dy)&\text{if }n=k.
\end{array}
\right.
\end{equation*}
Note that if $\mathcal{T}^{a,b}$ denotes the transformation defined by (\ref{eq:CensorTr}), then the transformation given by (\ref{eq:RcensorTr}) is the $k$-fold composition $\mathcal{T}^{a_{k-1},a_k}\circ\cdots\circ\mathcal{T}^{a_1,a_2}\circ\mathcal{T}^{a_0,a_1}$.
We then deduce that $(\mu^{\bf a}_t,t\geq0)$ is still a MRL family.
 
\section*{Acknowledgement}
We are grateful to the anonymous referee whose valuable suggestions and comments improved significantly the presentation of this paper. We also thank him for drawing our attention to the book of Müller and Stoyan \cite{MS}.

\end{document}